\documentclass[a4paper,11pt]{amsart}
\usepackage{amsmath,inputenc,euscript,amssymb,geometry}
\usepackage{graphicx}
\usepackage{amssymb}
\usepackage{latexsym}
\usepackage{amssymb,amsbsy,amsmath,amsfonts,amssymb,amscd,color}
\usepackage{tikz}
\usepackage{mathrsfs}
\usepackage{enumitem}
\usepackage{hyperref}

\DeclareFontFamily{U}{matha}{\hyphenchar\font45}
\DeclareFontShape{U}{matha}{m}{n}{
  <-6> matha5 <6-7> matha6 <7-8> matha7
  <8-9> matha8 <9-10> matha9
  <10-12> matha10 <12-> matha12
  }{}
\DeclareSymbolFont{matha}{U}{matha}{m}{n}
\DeclareMathSymbol{\Lt}{3}{matha}{"CE}

\newtheorem{lemma}{Lemma}[section]
\newtheorem{theorem}[lemma]{Theorem}
\newtheorem*{theorem*}{Theorem}
\newtheorem{prop}[lemma]{Proposition}

\newtheorem{remark}[lemma]{Remark}

\numberwithin{equation}{section}
\DeclareMathOperator{\curl}{curl}
\DeclareMathOperator{\divb}{div}

\begin{document}
\title{A measure\,-\,$L^\infty$ div-curl lemma}
\author[Valeria Banica]{Valeria Banica}
\address[Valeria Banica]{Sorbonne Universit\'e, CNRS, Universit\'e de Paris, Laboratoire Jacques-Louis Lions (LJLL), F-75005 Paris, France}
\email{Valeria.Banica@sorbonne-universite.fr}

\author[Nicolas Burq]{Nicolas Burq}
\address[Nicolas Burq]{Laboratoire de math\'ematiques d'Orsay, CNRS, Universit\'e Paris-Saclay, B\^at.~307, 91405 Orsay Cedex, France and Institut Universitaire de France}
\email{Nicolas.Burq@universite-paris-saclay.fr } 
\begin{abstract}In this note we give a very short proof of the div-curl lemma in the limit conjugate case $\mathcal M-L^\infty$, where $\mathcal{M}$ is the set of Radon measures on $\mathbb{R}^d$. The proof follows the classical approach by defining here the product in the sense of distributions via a non unique microlocal Hodge's decomposition. 
The result is valid for many other spaces than  $\mathcal M-L^\infty$, including the classical div-curl lemma spaces $L^p-L^{p'}$ for $1<p<\infty$, and spaces of non conjugated regularity.
\end{abstract}
\date\today
\maketitle
\section{Introduction}
Div-curl lemmas have been introduced by Murat and Tartar while establishing the compensated compactness theory in the end of the seventies (\cite{Mu1},\cite{Mu2},\cite{Ta1},\cite{Ta2}, see also \S 7,17 of \cite{Tartarbook}, and \cite{Mu3}).

The classical result states, for $1<p<\infty$ and $p'$ its conjugate, that given two sequences $(v_n)$ and $(w_n)$ uniformly bounded and weakly converging in $L^{p}$ and $L^{p'}$respectively, such that $(\divb\, v_n)$ is uniformly bounded in $L^p$ and $(\curl w_n)$ is uniformly bounded in $L^{p'}$, their product $(v_n\cdot w_n)$ converges in the sense of distributions\footnote{Another slightly different classical setting is the case, based mainly on the case $p=2$, where the sequences $(v_n)$ and $(w_n)$ are bounded and weakly converging in $L^{p}$ and $L^{p'}$respectively, and such that $(\divb\, v_n)$ and $(\curl w_n)$ are only compact in $W^{-1,p}$ and $W^{-1,p'}$ respectively.}. The classical div-curl lemma has several different proofs\footnote{In the case $p=2$ an approach based on default measures can be used. However these objects do not have an obvious extension to the $L^1/L^\infty$ regularity.}, see for instance \cite{Pr} for a review. Its classical proof uses Hodge's decomposition on divergence-free fields $w_n=y_n+\nabla z_n$, retrieving for instance uniform boundedness of $z_n$ in $W^{1,p'}$, which allows for compactness in $L^{p'}$. We note that here $1<p<\infty$, thus Hodge's decomposition on divergence-free fields is classical. We also note that in the conjugate-exponent spaces $L^p-L^{p'}$ there is no issue with defining the products $v_n\cdot w_n$.

A recent extension has been obtained by Briane, Casado-D\'iaz and Murat in \cite{Br} by considering for the regularity of $v_n$ and $w_n$, instead of $L^p-L^{p'}$, the spaces $L^p-L^q$ with $\frac 1p+\frac 1q=1+\frac 1d$ and $1<p,q<\infty$, and by considering for $\divb v_n$ and $\curl w_n$ some other appropriate regularity, see Theorem 2.3 in \cite{Br}. They also considered the spaces $\mathcal M-L^d$ and $L^d-\mathcal M$, where $\mathcal{M}$ is the set of Radon measures, see Theorems 3.1 and 4.1 respectively in \cite{Br}. In the first case $L^p-L^q$ the authors give a definition of the product $v_nw_n$ in the sense of distributions, based on Hodge's decompositions on divergence-free fields for both $v_n$ and $w_n$. In the remaining two cases, involving $\mathcal M$, firstly the product is defined by proving moreover an extension to measures of the representation of divergence free functions in $L^1$ of Br\'ezis and Van Schaftingen from \cite{BrVS}. 
Then, once the product is well defined, for instance in the first case $L^p-L^q$, the uniform boundedness of $z_n$ in $W^{1,q}$ does not allow for compactness in $L^{p'}$, since $W^{1,q}$ is embedded in $L^{q^*}=L^{p'}$ but without compactness. The authors use the defect of compactness of the embedding, with the help of the celebrated concentration compactness principle of Lions \cite{Li}, to get a $\mathcal D'$-limit for $v_nz_n$, involving $v\cdot w$ and a combination of Dirac measures. \\

The main aim\footnote{The motivation to look at div-curl lemmas came to us, after having obtained in \cite{BBWFL1} results of $L^1$-regularity in linear situations, by having in mind nonlinear situations.} of this note is to give a short proof for the limit conjugate case $\mathcal M-L^\infty$. We shall use a non unique microlocal Hodge's decomposition that, contrary to the classical one, does not involve a divergence free field part. It will be our main ingredient, that will allow us to define the products $v_n\cdot w_n$ and $v\cdot w$, and modifying slightly the standard approach to obtain very easily a limit in the sense of distributions for $v_n\cdot w_n$. The method applies to much more spaces than $\mathcal M-L^\infty$, as stated in Proposition \ref{prop1} and Theorem \ref{theo1}.

Let $d\geq 2$, $1\leq  p \leq +\infty$, and $\mathcal F$ a function space among $L^p$, Sobolev spaces and $\mathcal M_0$, the set of finite Radon measures. We denote  
\begin{equation}
\begin{aligned}
\mathcal{F}_{\curl}&:= \{ u \in \mathcal F( \mathbb{R}^d, \mathbb{R}^d), \curl u \in \mathcal F( \mathbb{R}^d, \mathbb{R}^{d^2})\},\\
\mathcal{F}_{\divb}&:= \{ u \in \mathcal F( \mathbb{R}^d, \mathbb{R}^d),\divb u \in \mathcal F( \mathbb{R}^d, \mathbb{R})\},\\
\end{aligned}
\end{equation} 
and we endow these spaces with their natural norms. For sake of simplicity we do not specify anymore the natural domain and range of the functions. The spaces $\mathcal{F}_{\divb}$ and $\mathcal{F}_{\curl}$ that we shall consider in the sequel are continuously stable by multiplication by a $\mathcal C^\infty_0$ localization.

\begin{remark} Since the conclusions in Proposition \ref{prop1} and Theorem~\ref{theo1} below are in $\mathcal{D}'$, thus local, we can weaken the global hypothesis and assume only local conditions as for instance
$ \mathcal{M}_{\divb}, (L^\infty_{\curl})_{\text{loc}}$ instead of $\mathcal{M}_{0,\divb},L^\infty_{\curl}$. In particular we could place ourselves in the case of an open set $\Omega$ of $\mathbb R^d$ instead of $\mathbb R^d$ and consider local spaces in $\Omega$. 
\end{remark}

We start with a result ensuring the existence of products in the sense of distributions between various spaces.

\begin{prop}\label{prop1} Let $ (v,w) \in \mathcal M_{0,\divb} \times L_{\curl}^\infty$ or belonging to one of  the following spaces:
\begin{enumerate}
\item $\mathcal M_{0,\divb} \times W_{\curl}^{-1+ \epsilon, \infty}$, for $\epsilon>0$, 
\item $\mathcal M_{0,\divb} \times W_{\curl}^{-1+\frac dp, p}$, for $d<p< \infty$, 
\item  $W^{-1+\epsilon,\infty}_{\divb} \times \mathcal M_{0,\curl}$, for $\epsilon>0$, 
\item  $W^{-\delta+\epsilon,p'}_{\divb} \times \mathcal M_{0,\curl}$, for $1< p\leq \frac d{d-1+\delta}$\footnote{The condition $1<p\leq \frac d{d-1+\delta}$ is equivalent to $\frac{d}{1-\delta}\leq p'< \infty$, thus varying $\delta$ in $(0,1)$ we have the range $d<p'< \infty$}. and $\delta\in(0,1),\epsilon>0$, 
\item  $W_{\divb}^{-\alpha,p'}\times W_{\curl}^{-1+\alpha, p}$, for $1< p< \infty $, $\alpha\in \mathbb R$,
\item  $W_{\divb}^{-\alpha,p'}\times W_{\curl}^{-1+\alpha+\epsilon, p}$, for $p\in\{1,\infty\}$, $\epsilon>0$, $\alpha\in \mathbb R$.
\end{enumerate}
Then $w$ admits a class of Hodge-type decompositions $w=y+\nabla z$ such that the following quantity
\begin{equation}\label{Hodgeprod}
(v\cdot w)_H:=v\cdot y -(\divb\,v)z+\divb\,(vz).
\end{equation}
is well-defined as a distribution, is independent of the choice of the decomposition, and coincides with the usual product if $(v,w) \in\mathcal C^\infty_0\times \mathcal C ^\infty_0$ or if $(v,w) \in L_{\divb}^{p'}\times L_{\curl}^p$ for $1\leq p\leq \infty$. 

Moreover, in all space $\mathcal{F}_{\divb}\times\mathcal{\tilde F}_{\curl}$ of the cases (ii)-(iv)-(v), thus not involving $L^\infty$ for which $C^\infty_0$ is not dense, the classical product map
$$ (v,w) \in\mathcal C^\infty_0\times \mathcal C ^\infty_0 \mapsto v \cdot w  \in \mathcal C ^\infty_0 $$
admits $(vw)_H$ as a unique continuous extension from  $\mathcal{F}_{\divb}\times\mathcal{\tilde F}_{\curl}$ to $\mathcal{D}' $. 
\end{prop}

\begin{remark}

In the remaining spaces $\mathcal{F}_{\divb}\times\mathcal{\tilde F}_{\curl}$ of the cases (i)-(iii)-(vi), involving $L^\infty$, the product $(v\cdot w)_H$ is still the natural one. More precisely, if $(v_n,w_n)$
is any sequence bounded in $\mathcal{F}_{\divb}\times\mathcal{\tilde F}_{\curl}$  converging in the sense of distributions to $(v,w)$, such that the product $v_n\cdot w_n$ is well-defined in the classical sense, for instance a mollified approximation of $(v,w)$, then we will obtain by Theorem \ref{theo1} that
$$(v\cdot w)_H\overset{\mathcal D'}{=}\lim_{n \rightarrow \infty}v_n\cdot  w_n,$$
and consequently, the product $(v\cdot w)_H$ is in this cases the unique extension from $\mathcal F_{\divb}\cap\mathcal C^0\times\mathcal {\tilde F}_{\curl}\cap\mathcal C^0$ to $\mathcal D'$ which enjoys this weak continuity property. 
\end{remark}

The proof of Proposition \ref{prop1} is based on the Hodge decomposition defined in \eqref{Hodge} and on the action of pseudo-differential operators on $L^p$ and $\mathcal M_0$. The definition, notations and properties of pseudo-differential operators that will be of use in this note are given in Appendix \ref{app}.

We note that the space from {\em{(v)}} with $\alpha=0$, i.e. $W_{\divb}^{p}\times W_{\curl}^{-1, p'}$ for $1< p<\infty$, extends the classical $L^p-L^{p'}$ framework. Moreover, as a subcase we obtain from it the space $L^p_{\divb}\times L^q_{\curl}$ with  $\frac 1p+\frac 1q=1+\frac 1d$, since\footnote{We recall some Sobolev embeddings that will be used in the sequel: $W^{k,\alpha}\subset L^p$ for $k\in(0,1), k<\frac d\alpha, \alpha\in [1,\infty), p\in[\alpha,\frac{d\alpha}{d-k\alpha}]$,  
$W^{k,\alpha}\subset \mathcal C^{0,\frac{k\alpha-d}{\alpha}}$ for $k\in (0,1), k>\frac d\alpha, \alpha\in [1,\infty]$, and $W^{\delta,1}\subset L^p$ for $p\in [1,\frac d{d-\delta}]$.}
 in this case $L^q=L^{(p^*)'}\subset W^{-1,p'}$. 
We also note that from {\em{(ii)}}, since $L^d\subset W^{-1+\frac{d}{d^+},d^+}$, we obtain as a subcase the space $\mathcal M_{0,\divb}\times L^d_{\curl}$, and similarly we obtain the space $L^d_{\divb}\times \mathcal M_{0,\curl}$ as a subcase of {\em{(iv)}}. These last three cases corresponds to the ones in \cite{Br} modulo the fact that here we have the same regularity for $v_n$ and $\divb v_n$ and for $w_n$ and $\curl w_n$. To extend to different appropriate regularities for $\divb v_n$ and $\curl w_n$ as in \cite{Br} the proof could be extended in the spirit of Remark \ref{remdecv}.

Now we can state the div-curl lemma. 
\begin{theorem}\label{theo1}
The map
$$(v,w) \in \mathcal{M}_{0,\divb}\times L^\infty_{\curl}\mapsto (v\cdot w)_H \in \mathcal{D'},$$
is weakly continuous. 
More precisely, let $(v_n,w_n)$ be  a bounded sequence in $\mathcal{M}_{0,\divb} \times L^\infty_{\curl}$ converging in the sense of distributions to $(v,w)$:
$$ v_n \overset{\mathcal{D}'}{\rightharpoonup} v, \qquad w_n \overset{\mathcal{D}'}{\rightharpoonup}  w,$$
so in particular $(v,w)\in  \mathcal{M}_{0,\divb}\times L^\infty_{\curl}$. 
Then we have the product convergence:
$$ (v_n\cdot w_n)_H \overset{\mathcal{D}'}{\rightharpoonup}  (v\cdot  w)_H,$$
where the products are defined by Proposition~\ref{prop1}.
The result is valid also with $\mathcal{M}_{0,\divb} \times L^\infty_{\curl}$ replaced by:\par
\,\hspace{3mm}{\it{(i) $\mathcal M_{0,\divb} \times W_{\curl}^{-1+ \epsilon, \infty}$, for $\epsilon>0$,}} \par
\,\hspace{1mm}{\it{(ii)$^*$ $\mathcal M_{0,\divb} \times W_{\curl}^{-1+\frac dp+\epsilon, p}$, for $d<p< \infty$, $\epsilon>0$,}}\par
\,\hspace{1mm}\it{(iii)}  $W^{-1+\epsilon,\infty}_{\divb} \times \mathcal M_{0,\curl}$, for $\epsilon>0$, \par
\,\hspace{2mm}\it{(iv)} $W^{-\delta+\epsilon,p'}_{\divb} \times \mathcal M_{0,\curl}$, for $1< p\leq \frac d{d-1+\delta}$ and $\delta\in(0,1),\epsilon>0$, \par
\,\hspace{2mm}\it{(v)$^*$}  $W_{\divb}^{-\alpha,p'}\times W_{\curl}^{-1+\alpha+\epsilon, p}$, for $1< p< \infty $, $\epsilon>0$, $\alpha\in \mathbb R$, \par
\,\hspace{2mm}\it{(vi)}  $W_{\divb}^{-\alpha,p'}\times W_{\curl}^{-1+\alpha+\epsilon, p}$, for $p\in\{1,\infty\}$, $\epsilon>0$, $\alpha\in \mathbb R$.

\end{theorem}

The proof of Theorem \ref{theo1} is based on Proposition \ref{prop1} and on  the fact that the smooth localization $u\in\ W^{\epsilon,p}\to \chi u \in L^p$ is compact for $\epsilon>0$ and $1\leq p\leq \infty$. Thus a bit of regularity is lost with respect to Proposition \ref{prop1}. This explains why Theorem \ref{theo1} is valid in all the functional settings listed in Proposition \ref{prop1} with slight modifications for the spaces in {\em{(ii)}} and {\em{(v)}}, i.e. for instance $W_{\divb}^{-\alpha,p'}\times W_{\curl}^{-1+\alpha, p}$, for $1< p<\infty $, $\alpha\in \mathbb R$, is replaced by $W_{\divb}^{-\alpha,p'}\times W_{\curl}^{-1+\alpha+\epsilon, p}$.

Theorem \ref{theo1} is not considering the spaces $W_{\divb}^{-\alpha,p}\times W_{\curl}^{-1+\alpha, p'}$, for $1< p<\infty$, and in particular the subset $L^p_{\divb}\times L^q_{\curl}$ with $\frac 1p+\frac 1q=1+\frac 1d, 1<p<\infty$. This is normal, since in this case concentration effects in terms of Dirac measures can appear in the limit of the product, as in Example 2.10 in \cite{Br}, and more generally in Theorem 2.1 in \cite{Br} that describes the limit of $v_nw_n$ in a more involved than simply $vw$. To get such a result here one could continue the analysis in the spirit of \cite{Br} by exploiting the defect of compactness of the embedding $W^{1,q}\subset L^{p'}$.  \\



{\bf{Aknowledgements: }} The authors thank Anne-Laure Dalibard, Gilles Francfort, Antoine Gloria and Fran\c{c}ois Murat for enriching discussions on the div-curl lemma. 
Both authors acknowledge the funding from the European Research Council (ERC) under the
European Union's Horizon 2020
research and innovation programme (Grant agreement 101097172 -- GEOEDP). The first author was also partially supported by the French ANR project BOURGEONS.

\section{Proof of Proposition~\ref{prop1}}

We shall use a Hodge decomposition as follows. 
 For $w\in \mathcal S '$, the space of temperate distributions, we define the operators $Z_\chi$ and $Y_\chi$ by 
 \begin{equation}\label{Hodge}
 Z_\chi w := -\chi(D)(-\Delta)^{-1}\divb w \in \mathcal D ', \quad Y_\chi w:=w-\nabla Z_\chi w\in\mathcal  D '
 \end{equation}
 where $\chi\in \mathcal C^\infty ( \mathbb{R}^d)$ is a localisation outside the origin that removes the singularity of the symbol of $(-\Delta)^{-1}$, i.e. $\chi$ is equal to $1$ on $\,^cB(\delta)$ for some $\delta>0$ and vanishes in a neighborhood of $0$. Thus $w$ decomposes as:
  \begin{equation}\label{Hodgedec}
 w=Y_\chi w+\nabla Z_\chi w,
  \end{equation}
 with $Y_\chi w$ that is not divergence free, contrary to the classical Hodge's decomposition,  and moreover it depends on $\chi$. 
 
 Usually one retrieves regularity for the free-divergence field part of the classical Hodge decomposition by an elliptic inversion argument, here we could do so also, but we simply note that
 $$Y_\chi w=w-\nabla Z_\chi w=w+\nabla \chi(D)(-\Delta)^{-1}\divb w=w+\chi(D)(-\Delta)^{-1}\nabla  \divb w$$
 $$=w+\chi(D)(-\Delta)^{-1}(\Delta+\curl\curl) \,w=(1-\chi(D))w+\chi(D)(-\Delta)^{-1}\curl \curl w.$$
 Thus, using pseudodifferential operators, whose definition, notations and properties that will be of use in this note are given in the Appendix \ref{app}, we have
$$
Y_\chi w\in \Psi^{-\infty} w+\Psi^{-1}\curl w.
$$
On the other hand, we shall retrieve regularity for $Z_\chi w$ from its definition since:
$$
 Z_\chi w=  -\chi(D)(-\Delta)^{-1}\divb \,w \in \Psi^{-1}w.
$$
Summing up we have obtained the following regularity correspondance between $w$ and its decomposition terms $Y_\chi w$ and $Z_\chi w$:
 \begin{equation}\label{YZreg} 
\left\{\begin{array}{c}Y_\chi w\in \Psi^{-\infty} w+\Psi^{-1}\curl w,\\  Z_\chi w\in \Psi^{-1}w.\end{array}\right.
  \end{equation}
  
  Now we are ready to define the product in the statement of Proposition \ref{prop1}. We use the Hodge decomposition \eqref{Hodgedec} of $w$ to define, whenever each product term in the following is well defined in the sense of distributions, the quantity:
\begin{equation}\label{prod}
(v\cdot w)_\chi =v\cdot Y_\chi w -(\divb\,v)Z_\chi w+\divb\,(v\,Z_\chi w).
\end{equation}
We  notice that  for $ (v,w) \in \mathcal C^\infty_0\times \mathcal C ^\infty_0$ or if $(v,w) \in L_{\divb}^{p'}\times L_{\curl}^p$ for $1\leq p\leq \infty$, this quantity identifies with the usual product:
$$v\cdot w=(v\cdot w)_\chi.$$

\begin{remark}\label{remdecv} Note that we have also the following regularity properties of $Y_\chi v$ and $Z_\chi v$:
 \begin{equation}\label{YZregbis} 
\left\{\begin{array}{c}Y_\chi w\in \Psi^{0} w,\quad \divb Y_\chi w\in \Psi^{-\infty} w,\\  Z_\chi w\in \Psi^{-2}\divb w,\end{array}\right.
  \end{equation}
that are useful if we have regularity information on $\divb w$ or if we need regularity on $\divb Y_\chi w$. Typically when the regularity is different for $v$ and $\divb v$, as in \cite{Br}, both decompositions of $v$ and $w$ are needed to define the product $vw$. 
In our case the Hodge decomposition \eqref{Hodge} is not on free divergence but on smooth divergence  fields, and the product formula \eqref{prod} can be extended to
\begin{equation}\label{prodbis}
v\cdot w =Y_\chi v\cdot Y_\chi w +\divb (Z_\chi v Y_\chi w)-(\divb Y_\chi w)\,Z_\chi v-(\divb Y_\chi v)\,Z_\chi w+\divb(Y_\chi v\,Z_\chi w)+\nabla Z_\chi v\cdot \nabla Z_\chi w.
\end{equation}
In this note we stick to the classical framework of same regularity for $v$ and $\divb v$ as well as for $w$ and $\curl w$, and we will need to decompose only $w$ and to use \eqref{YZreg}. 
\end{remark}

In the following we shall prove Proposition \ref{prop1}, case by case. We place ourselves first in the case of the space $\mathcal{M}_{0,\divb}\times L^\infty_{\curl}$. 
From \eqref{YZreg} and Proposition \ref{pseudo} (ii) on the action of pseudodifferential operators on $L^\infty$, we get that$$
Y_\chi w, Z_\chi w\in \Psi^{-1}L^\infty\subset W^{1^-,\infty}\subset \mathcal C^0.
$$
In particular each product term in the \eqref{prod}  is well defined in the sense of distributions, and moreover the product does not depend on the choice of the  function $\chi$ in the Hodge decomposition \eqref{Hodge}. Indeed, if $\tilde\chi$ is another cut-off as $\chi$, i.e. equal to $1$ outside a neigborhood of the origin and vanishing in a neighborhood of $0$.  then $\chi - \tilde\chi$ is compactly localized, thus  $(\chi - \tilde\chi)( D)\in\Psi^{-\infty}$ and we have

$$
\begin{gathered}
\delta: = Z_{\chi} w - Z_{\tilde\chi}w = (\chi - \tilde\chi)( D) (- \Delta)^{-1} \divb w \in C^\infty ,\\
  Y_{\chi} w - Y_{\tilde \chi} w= \nabla \delta  \in C^\infty,
  \end{gathered}
$$
and we deduce 
$$
 (v\cdot w)_{\chi}- (v\cdot w)_{{\tilde\chi} } =v\cdot  \nabla \delta - \divb (v)\, \delta  + \divb ( v\, \delta ) ,
$$
 which is equal to $0$ in the sense of distributions because $\delta $ is smooth. Thus, considering in the statement the class of Hodge decompositions $w=Y_\chi w+\nabla Z_\chi w$ from \eqref{Hodge}, the product $(v\cdot w)_H$ in \eqref{Hodgeprod} equals to $(vw)_\chi$, thus it is well-defined, independent of the choice of $\chi$ in the Hodge decomposition \eqref{Hodge}, and coincides with the usual product for $(v,w) \in \mathcal C^\infty_0\times \mathcal C ^\infty_0$. \\

We now turn to the others space cases of Proposition~\ref{prop1}. We consider first the $L^\infty_{\divb} \times \mathcal M_{0,\curl}$ case. The proof will go the same lines as before. For $w\in\mathcal C^\infty_0$,  from \eqref{YZreg} and Proposition \ref{pseudo} (iii) on the action of pseudodifferential operators of negative order on $\mathcal M_0$, we get that
$$Y_\chi w, Z_\chi w\in \Psi^{-1}\mathcal M_{0} \subset W^{1^-,1}\subset L^1,$$
and we conclude as before the proof of Proposition~\ref{prop1}. \\

Treating the cases:\\
(v)  $W_{\divb}^{-\alpha,p'}\times W_{\curl}^{-1+\alpha, p}$, for $1< p<\infty $, $\alpha\in \mathbb R$,\\
(vi)  $W_{\divb}^{-\alpha,p'}\times W_{\curl}^{-1+\alpha+\epsilon, p}$, for $p\in\{1,\infty\} $, $\epsilon>0$, $\alpha\in \mathbb R$.,\\
goes similarly, by retrieving the conjugate regularity from \eqref{YZreg} and Proposition \ref{pseudo} (i).

Treating the remaining cases goes also the same, by using moreover Sobolev embeddings and and Proposition \ref{pseudo} (ii)-(iii):\\
(i) $\mathcal M_{0,\divb} \times W_{\curl}^{-1+ \epsilon, \infty}$, for $\epsilon>0$, uses $\Psi^{-1}W^{-1+\epsilon,\infty}\subset W^{\epsilon,\infty}\subset \mathcal C^0$,\\
(ii) $\mathcal M_{0,\divb} \times W_{\curl}^{-1+\frac dp+ \epsilon, p}$, for $d<p< \infty$, $\epsilon>0$, uses
$$\Psi^{-1}W^{-1+\frac dp+ \epsilon, p}\subset W^{\frac dp+\epsilon,p}\subset \mathcal C^0,$$
(iii)$W^{-1+\epsilon,\infty}_{\divb} \times \mathcal M_{0,\curl}$, for $\epsilon>0$, uses
$$\Psi^{-1}\mathcal M_{0} \subset W^{1-\epsilon,1},$$
(iv) $W^{-\delta+\epsilon,p'}_{\divb} \times \mathcal M_{0,\curl}$, for $1< p\leq \frac d{d-1+\delta}$, $\delta\in(0,1), \epsilon>0$, uses
$$\Psi^{-1}\mathcal M_{0} \subset \Psi^{\epsilon-\delta}\Psi^{-1+\delta-\epsilon}\mathcal M_{0} \subset  \Psi^{\epsilon-\delta}W^{1-\delta,1}\subset \Psi^{\epsilon-\delta}L^p\subset W^{\delta-\epsilon,p}.$$

To prove the last extension result in Proposition \ref{prop1}, say for $\mathcal{M}_{0,\divb}\times W^{-1+\frac dp+\epsilon,p}_{\curl}$ with $d<p<\infty$ (the other cases follow similarly), we use the same arguments before to get that the map 
$$ (v,w) \in \mathcal C^\infty_0\times \mathcal C^\infty_0\mapsto v w  \in \mathcal{D}'$$
is continuous for the  $\mathcal{M}_{0,\divb}\times W^{-1+\frac dp+\epsilon,p}_{\curl}$ topology.
Indeed, as before, we get from \eqref{YZreg} that along with 
$$
Y_\chi w, Z_\chi w\in \Psi^{-1}W^{-1+\frac dp+\epsilon,p}\subset W^{\frac dp+\epsilon,p}\subset \mathcal C^0,
$$
we have also the continuity for the $W^{-1+\frac dp+\epsilon,p}_{\curl}$ topology of the maps 
$$
w\in C^\infty_0 \mapsto Y_\chi w  \in \mathcal C^0,\quad w\in C^\infty_0 \mapsto Z_\chi w  \in \mathcal C^0.
$$
As a consequence we deduce that the maps 
\begin{equation}\label{maps} 
 (v,w) \in \mathcal C^\infty_0\times \mathcal C^\infty_0  \mapsto \begin{cases}v\cdot Y_\chi w  \in \mathcal{D}',\\ (\divb v) Z_\chi w \in \mathcal{D}',\\
v\,Z_\chi w \in \mathcal{D}',
\end{cases}
\end{equation}
are continuous for the  $\mathcal{M}_{0,\divb}\times W^{-1+\frac dp+\epsilon,p}_{\curl}$ topology, thus so is the map 
$$ (w\cdot v) \in \mathcal C^\infty_0 \times \mathcal C^\infty_0 \mapsto w\cdot v =v\cdot Y_\chi w-(\divb\, v)Z_\chi w+\divb\,(v\,Z_\chi w) \in \mathcal{D}'.$$ 
Consequently the product map has a unique extension from $\mathcal{M}_{0,\divb}\times W^{-1+\frac dp+\epsilon,p}_{\curl}$ to $\mathcal D'$. This extension defines the product on $ \mathcal{M}_{0,\divb}\times W^{-1+\frac dp+\epsilon,p}_{\curl}$ as a distribution, and it does not depend on the choice of the cut-off function $\chi$ because it coincides, whatever this choice, with the usual product on the dense subspace $\mathcal C^\infty_0\times \mathcal C^\infty _0$.\\

\section{Proof of Theorem~\ref{theo1}}
We consider a Hodge decomposition as in \eqref{Hodgedec}:
$$w_n=Y_\chi w_n+\nabla Z_\chi w_n.$$
Let $\varphi\in\mathcal C^\infty_0$. To prove Theorem~\ref{theo1} , we study the convergence of the sequence
$$\langle (v_nw_n)_H,\varphi\rangle= \int v_nY_\chi w_n \varphi- (\divb\,v_n)Z_\chi w_n\varphi- v_nZ_\chi w_n\nabla \varphi.$$
Let $K$ be a compact subset of $\mathbb R^d$ containing the support of $\varphi$. Let $\psi \in  \mathcal C^\infty_0$ equal to $1$ on $K$ and let $\tilde\psi \in  \mathcal C^\infty_0$ equal to $1$ on $K$ and supported on the set where $\nabla\psi$ vanishes. 
By localizing we have
$$\tilde\psi\psi\, w_n=\tilde\psi\psi\, Y_\chi w_n+\tilde\psi\psi\, \nabla Z_\chi w_n$$
$$=\tilde\psi\psi\, Y_\chi w_n+\tilde\psi\,\nabla (\psi\, Z_\chi w_n)-\tilde\psi\,Z_\chi w_n\nabla\psi=\tilde\psi\psi\, Y_\chi w_n+\tilde\psi\,\nabla (\psi\, Z_\chi w_n). $$
%
\medskip

We consider first the case $\mathcal M_{0,\divb}\times L^\infty_{\curl}$. 
Since by hypothesis $\{w_n\}$ is a bounded sequence in $L^\infty_{\curl}$ it follows from \eqref{YZreg}  that $\{Y_\chi w_n\}$ and $\{Z_\chi w_n\}$ are bounded sequence in $W^{1-\epsilon,\infty}$. Therefore $\{\tilde\psi\psi\, Y_\chi w_n\}$ and $\{\psi\, Z_\chi w_n\}$  are precompact in $W^{1-2\epsilon,\infty}$ and thus also in $\mathcal C^0$. Up to a subsequence, for which we drop the subsequence indices for simplicity, we have the existence of continuous limits $y$ and $z$:
$$\tilde\psi\psi\,  Y_\chi w_n\overset{\mathcal C^0}{\longrightarrow} y,\quad \psi\, Z_\chi w_n\overset{\mathcal C^0}{\longrightarrow} z,$$
with values in $K$ independent of $\psi,\tilde\psi$, that we use to define two functions globally in space, that we still call $y$ and $z$.

We note that we have obtained convergence in $\mathcal D'(K)$ of a subsequence of $w_n=Y_\chi w_n+\nabla Z_\chi w_n$ to $y+\nabla z$. As by hypothesis $w_n$ converges in the sense of distributions to $w$ it follows that 
$$w\overset{\mathcal D'(K)}{=}y+\nabla z.$$

Finally, we note that the boundeness of $\{v_n\}$ in $\mathcal M_{0,\divb}$ and its convergence to $v$ in the sense of distributions imply the weak $\mathcal M_0-$convergence of $\{v_n\}$ and $\{\divb v_n\}$ to $v$ and $\divb v$ respectively. Thus, by using $\varphi=\tilde \psi\psi\varphi$, $\varphi=\psi\varphi$, $\nabla\varphi=\psi\nabla\varphi$ and the strong convergences in $\mathcal C^0$ of $\tilde\psi\psi\,  Y_\chi w_n$ and $\psi\, Z_\chi w_n$, we obtain:
$$\langle (v_nw_n)_H,\varphi\rangle= \int v_nY_\chi w_n \varphi- (\divb\,v_n)Z_\chi w_n\varphi- v_nZ_\chi w_n\nabla \varphi$$
$$=\int v_n(\tilde\psi\psi \, Y_\chi w_n)\varphi- (\divb\,v_n)(\psi Z_\chi w_n)\,\varphi- v_n(\psi Z_\chi w_n)\nabla \varphi$$
$$\overset{n\rightarrow\infty}{\longrightarrow} \int vy\varphi- (\divb\,v)z\varphi- vz\nabla \varphi.$$
Now we will show that this last quantity is precisely $\langle (vw)_H,\varphi\rangle$, by proving $y=Y_\chi w, z=Z_\chi w$. Indeed, for a test function $\phi\in\mathcal C^\infty_0(K)$ we have:
$$\langle Z_\chi w-z,\phi\rangle
=\langle Z_\chi w-Z_\chi w_n,\phi\rangle+\langle Z_\chi w_n-z,\phi\rangle$$
$$=\langle (-\chi(D)(-\Delta)^{-1}\divb)(w-w_n), \phi\rangle +o(1)$$
$$=\langle w-w_n, (-\chi(D)(-\Delta)^{-1}\divb)^*\phi\rangle +o(1)=o(1),$$
since $(-\chi(D)(-\Delta)^{-1}\divb)^*\phi\in \mathcal S\subset L^1$, $\mathcal C^\infty_0$ is dense in $L^1$, $\{w_n\}$ is bounded in $L^\infty$ and $\{w_n\}$ converges to $w$ in $\mathcal D'$. Thus $z=Z_\chi w$ in $\mathcal D'(K)$ and therefore we also have $y=w-\nabla z=w-\nabla Z_\chi w=Y_\chi w$ in $\mathcal D'(K)$. Since they are continuous functions and by varying $K$ we obtain that $y=Y_\chi w, z=Z_\chi w$.

Therefore we have proved the existence of a subsequence $(v_{n_k}w_{n_k})_H$ converging to $(vw)_H$ in the sense of distributions. This implies that the result is valid on the whole sequence because there is only one possible limit. Thus Theorem \ref{theo1} is proved in the case $\mathcal M_{0,\divb}\times L^\infty_{\curl}$.

The other cases of spaces can be treated similarly, by using Proposition \ref{prop1} and  the fact that the smooth localization $u\in\ W^{\epsilon,p}\to \chi u \in L^p$ is compact for $\epsilon>0$ and $1\leq p\leq \infty$.\\

\appendix
\section{}\label{app}
In this appendix we recall the definition of pseudodifferential operators and describe their action on $L^p$ spaces and on $\mathcal M_0$. We refer to~\cite[Chapter XVIII]{Ho} for a general presentation or~\cite[\S VI.6]{Stein} for a presentation closer to our needs of the pseudodifferential calculus (see also \cite{AlGe}).
Let us first recall the definitions of the class of symbol of order $\delta\in\mathbb R$: 
$$
 S^\delta( \mathbb{R}^d) =\{a\in C^\infty( \mathbb{R}^{d});  \\
 \forall \alpha, \beta \in \mathbb{N}^d, \sup_{x,\xi \in \mathbb{R}^d}  |\partial_x^\alpha  \partial_\xi^\beta a(x, \xi)| (1+ |\xi|)^{-\delta+|\gamma|}=: \|a\|_{\delta, \alpha, \beta}<+\infty\}.
$$
To any symbol $a \in S^\delta_{\text{cl}}(\mathbb{R}^d)$ we can associate an operator on the temperate distributions set $\mathcal{S}'( \mathbb{R}^d)$ by the formula
$$ a(x,D_x) u (x)= \text{Op} (a) u(x)= \frac{ 1 }{(2\pi)^d} \int e^{i(x-y) \cdot \xi} a(x, \xi) u(y) dy d\xi.$$
The set of such operators, that are called pseudodifferential operators of order $\delta$, is denoted by $\Psi^\delta$. The set $\Psi^{-\infty}$ is defined as $\cap_{\delta<0}\Psi^\delta$.  In the following proposition we gather the results needed in this note regarding to the action of pseudodifferential operators.

\begin{prop}\label{pseudo}
Let $A = Op(a) \in \Psi^\delta_{\text{cl}}$. Then $A$ acts continuously:
\begin{enumerate}
\item for $1< p < +\infty$, from $W^{s,p}( \mathbb{R}^d)$ to $W^{s-\delta,p}( \mathbb{R}^d)$,
\item for $p \in\{1, +\infty\}$ and $\epsilon>0$, from $W^{s,p}( \mathbb{R}^d)$ to $W^{s-\delta-\epsilon,p}( \mathbb{R}^d)$,
\item for $\delta<0$ and $\epsilon>0$, from $\mathcal M_0(\mathbb R^n)$ to $W^{-\delta-\epsilon,1}(\mathbb R^n)$.
\end{enumerate}
\end{prop}

\begin{proof}
The results (i)-(ii) are classical: for the first one see for instance \cite[Section VI.5.2]{Stein}, and the second follows by the dual estimate of the Lemma in~\cite[Section VI.5.3.1]{Stein}. The result (iii) is less classical and we give here a complete (simple) proof. 
We use a dyadic partition of unity 
$$ 1 =\sum_{j\geq 0} \phi_j (\xi) , \, \phi_0 \in C^\infty _ 0 (\mathbb{R}^d), \,\forall j \geq 1, \phi_j (\xi) =  \phi( 2^{-j} \xi), \,\phi \in C^\infty_0 ( \{ \frac 1 2 < \| \xi\| <2 \} ),$$
 to decompose
 $$  A  = \sum_{j\geq 0}   A\,  \phi_j (D).$$
 Each operator of the sum has the kernel
$$
K_j(x,y):=  \frac 1 {(2\pi)^d} \int e^{i(x-y) \cdot \xi }  a(x,\xi) \phi_j(\xi)  d\xi.$$
In view of the decay of $a$ and of the localization of $\phi_j$ we have
$$|K_j(x,y)|\leq C 2^{j(d+\delta)}\|a\|_{\delta,0,0}.$$
Also, integrating by parts $N$ times using the identity
$$ L(e^{i(x-y) \cdot \xi} )= - e^{i(x-y) \cdot \xi}, \qquad  L= \frac {i (x-y) \cdot \nabla_\xi} {\| x-y\|^2} ,$$
we get for $N>1$ if $\|x-y\|\neq 0$, that
$$
|K_j(x,y)|=\Big|  \frac {1} {(2\pi)^d} \int e^{i(x-y) \cdot \xi } L^N (a(x,\xi) \phi_j(\xi) ) d\xi\Big| \leq C 2^{j(d+\delta)}\frac{\|a\|_{\delta,0,N}}{(2^j\|x-y\|)^N}.
$$
Therefore by taking $N=d+1$ we obtain
$$\int|K_j(x,y)|dx=\int_{2^j\|x-y\|<1}|K_j(x,y)|dx+\int_{2^j\|x-y\|>1}|K_j(x,y)|dx
\leq C 2^{j\delta}(\|a\|_{\delta,0,0}+\|a\|_{\delta,0,N}),
$$
and similarly 
$$\int|K_j(x,y)|dy\leq C(a)2^{j\delta}.$$
For $\mu\in\mathcal M_0$ we have, since $\delta<0$,
$$\|A\mu\|_{L^1}\leq \sum_{j\geq 0}\Big\|\int K_j(x,y)d\mu(y)\Big\|\leq C(a) \|\mu\|\sum_{j\geq 0}2^{j\delta}\leq C(a)\|\mu\|.$$
In conclusion $\Psi^\delta$ acts continuously from $\mathcal M_0(\mathbb R^n)$ to $L^1(\mathbb R^n)$ for any $\delta<0$. Combining with (ii) we obtain that $\Psi^\delta$ acts continuously from $\mathcal M_0(\mathbb R^n)$ to $W^{-\delta-\epsilon,1}(\mathbb R^n)$ for any $\delta<0$ and $\epsilon>0$.
\end{proof}

\end{document}